\documentclass[10pt,a4paper]{article}

\author{R. R. de Araujo\thanks{Instituto Federal de Educa\c c\~ao, Ci\^encia e Tecnologia de S\~ao Paulo, CEP 11533-160, Brazil, robson.ricardo@ifsp.edu.br} \qquad F. C. Polcino Milies\thanks{Instituto de Matem\'atica e Estat\'istica, Universidade de S\~ao Paulo, Caixa Postal 66281, CEP-05315-970,  
 Brazil, polcino@ime.usp.br.} \qquad R. A. Ferraz\thanks{Instituto de Matem\'atica e Estat\'istica, Universidade de S\~ao Paulo, Caixa Postal 66281, CEP-05315-970,  
 Brazil, raul@ime.usp.br.}}
\title{Minimal group codes over alternating groups}
\date{}

\usepackage[latin1]{inputenc}
\usepackage[T1]{fontenc}
\usepackage{syntonly}
\usepackage{verbatim}
\usepackage{makeidx}
\usepackage{latexsym}
\usepackage{graphicx,color}
\usepackage{graphicx}
\usepackage{amsthm,amsfonts}
\usepackage{ulem}
\usepackage{amssymb}
\usepackage{amsmath}
\usepackage{hyperref}
\usepackage{lipsum}
\usepackage[all,cmtip]{xy}
\usepackage{mathtools}

\usepackage[vcentermath]{youngtab}

\newtheorem{teorema}{Theorem}[section]
\newtheorem{ex}{Example}[section]
\newtheorem{corol}{Corollary}[section]

\newtheorem{defin}{Definition}[section]

\newtheorem{obs}{Remark}[section]

\def\linha#1{%
  \hbox to \hsize{%
      \vbox{\centering #1}}%
      \vspace{4mm}}

\begin{document}

\maketitle

\begin{abstract} In this work we show that every minimal code in a semisimple group algebra $\mathbb{F}_qG$ is essential if $G$ is a simple group. Since the alternating group $A_n$ is simple if $n=3$ or $n\geq 5$, we present some examples of minimal codes in  $\mathbb{F}_qA_n$. For this purpose, if $char(\mathbb{F}_q)> n$, we present the Wedderburn-Artin decomposition of $\mathbb{F}_qS_n$ and $\mathbb{F}_qA_n$ and explicit some of the centrally primitive idempotents of $\mathbb{F}_qS_n$ and $\mathbb{F}_qA_n$.
\end{abstract}

\thanks{The authors were partially supported by FAPESP, Proc. 2015/09162-9.}

\section{Introduction}

Let $\mathbb{F}_q$ be a finite field with $q$ elements and $G$ be a finite group. A \textit{group code} is any ideal of the group algebra $\mathbb{F}_qG$. Group codes were introduced independently by S. D. Berman and  MacWilliams \cite{berman1,berman2,macwi} and since then they have extensively been used to obtain interesting error-correcting codes \cite{bernal2009, essen2018, idem17, frp, ferraz07, left}.

A minimal code is any minimal ideal in the algebra $\mathbb{F}_qG$. If the characteristic of $\mathbb{F}_q$ not divides the order of $G$, the group algebra $\mathbb{F}_qG$ is semisimple. In this case $\mathbb{F}_qG$ has a complete set of centrally primitive idempotents, which are the generators of the minimal ideals in this algebra. In \cite{idem17} the authors define the concept of \textit{essential idempotent} and investigate its use in coding theory. A primitive idempotent $e\in \mathbb{F}_qG$ is said to be essential if $e\hat{H}=0$ for every subgroup $H\neq \{1\}$ in $G$, where $\hat{H}=|H|^{-1}\sum_{h\in H}  h$. A group code is said to be an \textit{essential code} if it is generated by an essential idempotent. In \cite{idem17} it is shown that minimal codes in $\mathbb{F}_qG$ are repetition codes whenever $G$ is an abelian non-cyclic group. Also, it is proved that, if $C$ is a code in $\mathbb{F}_{2^f}G$ of dimension $k$ and length $n=2^k-1$, the code $C$ is essential if and only if $C$ is a simplex code. In \cite{essen2018} it is shown that a binary linear code of constant weight are either generated by an essential idempotent or is a repetition code.

In this work we show that every minimal code in a semisimple group algebra $\mathbb{F}_qG$ is essential if $G$ is a simple group. Important examples of simple groups are given by the alternating groups $A_n$, for $n=3$ or $n\geq 5$. Motivated by the fact that all minimal codes in $\mathbb{F}_qA_n$ are essential if $n=3$ or $n\geq 5$ and $char(\mathbb{F}_q)>n$, we present a study about some minimal codes in this algebra. Adapting the work of \cite{giambruno} for finite fields, we present the Wedderburn-Artin decomposition of the semisimple algebra $\mathbb{F}_qA_n$ and describe some centrally primitive idempotents of it. For this purpose, firstly we also present the Wedderburn-Artin decomposition and give a complete set of centrally primitive idempotents of the semisimple algebra $\mathbb{F}_qS_n$. In the end we explicit the parameters of all minimal ideals in $\mathbb{F}_5A_3$, $\mathbb{F}_7A_3$, $\mathbb{F}_5A_4$ and $\mathbb{F}_7A_4$ and present the generator matrices of some of them. We highlight that a minimal code of $\mathbb{F}_5A_3$ is the best linear code with length $n=3$ and $k=2$ over $\mathbb{F}_5$. In turn, the other non-trivial minimal codes obtained in this work have not the best known minimum distance but they are not bad.

This paper is organized as follows. In Section \ref{sec1} we present known definitions and facts about group algebras, group codes and symmetric groups. In Section \ref{sec2} we give the Wedderburn-Artin decomposition of $\mathbb{F}_qS_n$ and $\mathbb{F}_qA_n$ in the semisimple case, present a complete set of centrally primitive idempotents of $\mathbb{F}_qS_n$ and explicit some centrally primitive idempotents of $\mathbb{F}_qA_n$. Finally, in Section \ref{sec3} we present some results and examples about minimal codes in $\mathbb{F}_qA_n$.

\section{Preliminaries}\label{sec1}

In this section, firstly we present some basic concepts about finite group algebras, codes and group codes. Secondly we give some definitions and facts about the representation of the symmetric group $S_n$.
\

\subsection{Group algebras and group codes}

In this work, consider $\mathbb{F}=\mathbb{F}_q$ a finite field with $q=p^f$ elements, where $p$ is a prime number and $f\geq 1$ is a natural number, and $G$ a finite group. A \textit{group algebra} is a ring of the form $\mathbb{F}G =\left\{ \sum_{g\in G} a_gg~|~a_g\in \mathbb{F} \right\}$ where the sum and the product are given, respectively, by
$$\sum_{g\in G} a_gg + \sum_{g\in G} b_gg = \sum_{g\in g} (a_g+b_g)g$$
and
$$\left(\sum_{g\in G} a_gg \right)\left(\sum_{g\in G} b_gg \right) = \left(\sum_{g,h\in G} a_gb_hgh \right).$$
If the characteristic of $\mathbb{F}_q$ does not divide the order of $G$, Maschke's Theorem guarantees that the group algebra $\mathbb{F}G$ is semisimple and Wedderburn-Artin Theorem says that
\begin{equation}
\mathbb{F}G\cong M_{n_1}(\mathbb{F}_{(1)})\oplus M_{n_2}(\mathbb{F}_{(2)})\oplus\ldots \oplus M_{n_s}(\mathbb{F}_{(s)}),
\end{equation}
where $\mathbb{F}_{(i)}$ is a extension field of $\mathbb{F}$ and $M_{n_i}(\mathbb{F}_{(i)})$ is the set of all square matrices of order $n_i$ over $\mathbb{F}_{(i)}$ (see \cite[Section 3.4]{milies2002}). Besides, for $i\in\{1,2,\ldots,s\}$, there exists a centrally primitive idempotent $e_i$ such that $M_{n_i}(\mathbb{F}_{(i)})\cong \mathbb{F}Ge_i$.

In coding theory, a \textit{$q$-ary linear code} of length $n$ and dimension $k$ is a $k$-dimensional vector subspace $C$ of $\mathbb{F}_q^n$. If there exists a group $G$ of order $n$ such that $C=\sigma(I)$ for some (left) ideal $I$ of $\mathbb{F}_qG$ and some isomorphism $\sigma:\mathbb{F}_qG\longrightarrow\mathbb{F}_q^n$ which maps $G$ to the standard basis of $\mathbb{F}^n$, we say that $C$ is a (left) \textit{group code}. For example, a cyclic code $C$ of length $n$ over $ \mathbb{F}_q$ (a linear code is cyclic if $(c_2,\ldots,c_n,c_1)\in C$ for all $(c_1,c_2,\ldots,c_n)\in C$) is a group code because it can be viewed as an ideal of $\mathbb{F}_qC_n$, where $C_n=\langle x|x^n=1\rangle$ is the cyclic group of order $n$. A criterion to decide when a linear code is a group code can be seen in \cite{bernal2009}.

Every (left) ideal $I$ of $\mathbb{F}_qG$ provides a group code and, so, $I$ can be called (left) \textit{group code} too. If $G$ is abelian (cyclic), we say that $I$ is an abelian (cyclic) group code. A group code $I$ in $\mathbb{F}_qG$ is said to be \textit{minimal} if $I$ is a minimal bilateral ideal of the group algebra $\mathbb{F}_qG$. The length $n$ and the dimension $k$ of the group code $I\lhd \mathbb{F}_qG$ are given, respectively, by the order of the group $G$ and the dimension of $I$ as a $\mathbb{F}_q$-subspace of $\mathbb{F}_qG$.

We define the \textit{Hamming distance} between two elements $\alpha = \sum_{g\in G} \alpha_gg$ and $\beta = \sum_{g\in G}\beta_g g$ of $\mathbb{F}_qG$, with $\alpha_g$ and $\beta_g$ in $\mathbb{F}_q$ for all $g\in G$, to be the number $d(\alpha,\beta)=|\{g\in G~|~\alpha_g \neq \beta_g \}|$. The \textit{Hamming weight} or \textit{minimum distance} of a group code $I\lhd \mathbb{F}_qG$ is denoted by $d=w(I)$ and is given by
\begin{equation}
w(I) := \min \{ d(\alpha,0)~|~\alpha\neq 0, ~\alpha\in I\}.
\end{equation}
The length $n$, the dimension $k$ and the Hamming weight $d$ of a group code $I$ in $\mathbb{F}_qG$ are their \textit{parameters}. In this case, we can say that $I$ is a $q$-ary $[n,k,d]$-code.

\subsection{The symmetric group}

In the following we summarize some concepts and results about partitions, tableaux and the symmetric group $S_n$. A detailed theory about this can be found in \cite{james}.

Let $n$ be a positive integer number. A \textit{partition} of $n$ is a finite sequence $\lambda=(\lambda_1,\lambda_2,\ldots,\lambda_k)$ such that $\sum_{i=1}^k \lambda_i = n$ and $\lambda_1\geq \lambda_2\geq \ldots \geq \lambda_k>0$. The finite set of all partitions of $n$ is denoted by $P(n)$. Every partition $\lambda=(\lambda_1,\lambda_2,\ldots,\lambda_k) \in P(n)$ can be associated with a \textit{Young diagram} $D_\lambda$, which is a diagram given by the concatenation of squares arranged in $k$ rows $r_i$, $i\in \{1,2,\ldots,k\}$, each of them with $\lambda_i$ squares. Filling the squares of $D_\lambda$ with distinct numbers of $\{1,2,\ldots,n\}$ someone obtains a structure called \textit{Young tableau}, denoted by $T_\lambda = T_\lambda(a_{ij})$, where $a_{ij}$ denotes the number of $\{1,2,\ldots,n\}$ put inside the square of row $i$ and column $j$. The transposition of $D_\lambda$ produces a new Young diagram, which is associated with other partition $\lambda^\prime=(\lambda_1^\prime,\lambda_2^\prime,\ldots,\lambda_l^\prime)\in P(n)$. The partition $\lambda^\prime$ is called the \textit{conjugate partition} of $\lambda$.

\begin{ex} Consider $n=9$ and the partition $\lambda=(3,2,2,2)\in P(9)$. Below it is presented the Young diagram $D_\lambda$ and an example of Young tableau $T_\lambda$:
	\begin{equation}\label{lambda9}
	D_\lambda = \yng(3,2,2,2) \qquad T_\lambda = \young(187,63,24,95)
	\end{equation}
	The transpose diagram of $D_\lambda$ is
	$$
	D_{\lambda^\prime} = \yng(4,4,1)$$
	which implies that the conjugate partition of $\lambda$ is $\lambda^\prime = (4,4,1) \in P(9)$.
\end{ex}

In a Young tableau $T_\lambda$, if the numbers of each row and of each column are arranged in increasing order, $T_\lambda$ is called a \textit{standard Young tableau} of $\lambda$. The number of standard Young tableaux associated to $\lambda$ is denoted by $d_\lambda$. For instance, the Young tableau $T_\lambda$ in Example \ref{lambda9} is not standard because the numbers in the first row are not arranged in increasing order.

Let $D_\lambda$ a Young diagram associated to $\lambda\in P(n)$. For the square in the position $i\times j$ of $D_\lambda$ we define the \textit{hook number} of the position $i\times j$ by $h_{ij}(\lambda)=\lambda_i+\lambda_j^\prime-(i+j)+1$. The \textit{hook number formula} is a known equality that expresses $d_\lambda$ in function of all $h_{ij}(\lambda)$ as following:
\begin{equation}
d_\lambda = \frac{n!}{\prod h_{ij}(\lambda)}.
\end{equation}

\begin{ex} Consider $\lambda=(3,2,2,2)\in P(9)$. A standard Young tableau associated to $\lambda$ is the following:
	$$
	T_\lambda(a_{ij})=\young(135,24,86,97)
	$$
	Since $\lambda^\prime=(4,4,1)$, the hook number associated to the position $a_{12}=3$ is $h_{12}(\lambda)=3+4-(1+2)+1=5$. Note that $h_{12}(\lambda)$ corresponds to the number of squares to the right of $3$ plus the number of squares below $3$ plus 1. By the hook number formula, the number of standard Young tableaux associated to $\lambda$ is $d_\lambda = 9!/(6.5.1.4.3.3.2.2.1) = 84$.
\end{ex}

Let $S_n$ be the symmetric group. We say that a permutation $\pi\in S_n$ have \textit{cycle type} $\lambda=(\lambda_1\,\lambda_2,\ldots,\lambda_k)\in P(n)$ if the cycles of $\pi$ have length $\lambda_1$, $\lambda_2$, $\ldots$, $\lambda_k$. For instance, in $S_5$ the permutation $\pi=(1~4~5)(2~3)$ has cycle type $\lambda=(3,2)$. There is a one-to-one correspondence among the conjugation classes of $S_n$, the permutations in $S_n$ of same cycle type and the ordinary irreducible representations of $S_n$. Moreover, each partition $\lambda\in P(n)$ is associated to a conjugation class $\mathcal{C}_{\lambda}^{S_n}$ and to a $S_n$-character $\chi_\lambda^{S_n}$ such that $\chi_\lambda^{S_n}(1)=d_\lambda$.

\section{Semisimple group algebras $\mathbb{F}_qS_n$ and $\mathbb{F}_qA_n$}\label{sec2}

Consider $\mathbb{F}_q=\mathbb{F}_{p^f}$ a finite field of characteristic $p>0$ and $n$ a positive integer number such that $p>n$. In this case, the group algebras $\mathbb{F}_qS_n$ and $\mathbb{F}_qA_n$ are semisimple because $p$ does not divides $n!$ nor $n!/2$. In this section we present the Wedderburn-Artin decomposition of $\mathbb{F}_qS_n$ and $\mathbb{F}_qA_n$, give a complete set of centrally primitive idempotents of $\mathbb{F}_qS_n$ and some centrally primitive idempotents of $\mathbb{F}_qA_n$. The theory presented here on $\mathbb{F}_qS_n$ and $\mathbb{F}_qA_n$ are obtained adapting results on $\mathbb{Q}S_n$ and $\mathbb{Q}A_n$ available in \cite{james, giambruno}.


\subsection{Decomposition and centrally primitive idempotents in $\mathbb{F}_qS_n$}
Consider a partition $\lambda=(\lambda_1,\lambda_2,\ldots,\lambda_k)\in P(n)$. Let $T_\lambda=T_\lambda(a_{ij})$ be a Young tableau associated to $\lambda$ and $\lambda^\prime=(\lambda_1^\prime,\lambda_2^\prime,\ldots,\lambda_l^\prime)$ be the conjugate partition of $\lambda$. Define the \textit{row stabilizer} and the \textit{column stabilizer} of $T_\lambda$, respectively, as
$$R_{T_\lambda} = \prod_{i=1}^k S_{\lambda_i}(a_{i1},\ldots,a_{i\lambda_i})
\qquad and \qquad C_{T_\lambda} = \prod_{i=1}^l S_{\lambda_i^\prime}(a_{1i},\ldots,a_{\lambda_i^\prime i})
$$
where $S_r(b_1,\ldots,b_r)$ denotes the set of all permutation of the sequence $(b_1,\ldots,b_r)$. It follows from \cite[Equation 1.5.4]{james} that
\begin{equation}\label{idempotente}
e_{T_\lambda} = \frac{d_\lambda}{n!}\sum_{p\in R_{T_\lambda}}\sum_{q\in C_{T_\lambda}} sgn(q)qp
\end{equation}
is a primitive idempotent of $\mathbb{F}_qS_n$, where $sgn(q)$ denotes the signature of the permutation $q$. Also, denote
\begin{equation}
\overline{e}_{T_\lambda} = \frac{d_\lambda}{n!}\sum_{p\in R_{T_\lambda}}\sum_{q\in C_{T_\lambda}} sgn(q)pq
\end{equation}
which is a primitive idempotent of $\mathbb{F}_qS_n$ too.

The Wedderburn-Artin decomposition of $\mathbb{F}_qS_n$ is well-known. It is given by
\begin{equation}\label{decSn}
\mathbb{F}_qS_n = \bigoplus_{\lambda \in P(n)} J_\lambda
\end{equation}
where $J_\lambda$ is a bilateral ideal isomorphic to $M_{d_\lambda}(\mathbb{F}_q)$, for all $\lambda\in P(n)$ (see \cite{james, giambruno}). In the following theorem we explicit the complete set of centrally primitive idempotents of $\mathbb{F}_q$, that is, the generators of $J_\lambda$. The proof is similar to that made in \cite[Lemma3]{giambruno} in the case of $\mathbb{Q}S_n$. In the following, consider $T_1,T_2,\ldots,T_{d_\lambda}$ the standard tableaux associated to $\lambda$ ordered by an order that satisfies the property
\begin{equation}
    i>j \Longleftrightarrow T_i>T_j \Longleftrightarrow T_i^\prime < T_j^\prime \Longleftrightarrow e_{T_i}e_{T_j}=0~and~\overline{e}_{T_i^\prime}\overline{e}_{T_j^\prime}=0,
\end{equation}
where $T_i^\prime$ denote the standard tableau associated to $\lambda^\prime$. Consider $J_\lambda\simeq M_{d_\lambda}(\mathbb{F}_q)$ and $J_{\lambda^\prime}\simeq M_{d_\lambda}(\mathbb{F}_q)$ the ideals corresponding to $\lambda$ and $\lambda^\prime$ in the decomposition (\ref{decSn}) of $\mathbb{F}_qS_n$. Denote by $e_\lambda^{S_n}$ and $e_{\lambda^\prime}^{S_n}$ the centrally primitive idempotents of $\mathbb{F}_qS_n$ that are generators of $J_\lambda$ and $J_\lambda^\prime$, respectively. Then:

\begin{teorema}\label{lemaorto} The centrally primitive idempotents of $\mathbb{F}_qS_n$ associated of $\lambda\in P(n)$ and $\lambda^\prime$, respectively, are given by
\begin{equation}\label{ideSN}
e_\lambda^{S_n}=1-\prod_{i=1}^{d_\lambda} (1-e_{T_{i}})\qquad and \qquad e_{\lambda^\prime}^{S_n}=1-\prod_{i=1}^{d_\lambda} (1-\overline{e}_{T_{i}^\prime}).
\end{equation}
\end{teorema}
\begin{proof} We prove the equality on left in (\ref{ideSN}). The proof of the other equality is analogous and it will be omitted. Let $e_\lambda^{S_n}$ denote the centrally primitive idempotent of $\mathbb{F}_qS_n$ associated to $\lambda$. Since $\oplus_{i=1}^{d_\lambda} \mathbb{F}_qS_n e_{T_i}=J_\lambda=\mathbb{F}_qS_n e_\lambda^{S_n}$, there exist $a_1,\ldots,a_{d_\lambda}\in \mathbb{F}_qS_n$ such that
\begin{equation}\label{auxiliar}
    e_\lambda^{S_n} = \sum_{i=1}^{d_\lambda}  a_ie_{T_i}.
\end{equation}
Multiplying (\ref{auxiliar}) on the right by $e_{T_1}$ we obtain $e_\lambda^{S_n}e_{T_1} = a_1e_{T_1}$. Since $e_{T_1}\in \mathbb{F}_qS_n e_\lambda^{S_n}$, then $e_{T_1}=\alpha_1 e_\lambda^{S_n}$ for some $\alpha_1\in \mathbb{F}_q S_n$. As $e_\lambda^{S_n}$ is a central idempotent, we get $e_{T_1}=\alpha_1 e_\lambda^{S_n}= \alpha_1 (e_\lambda^{S_n})^2 = e_\lambda^{S_n}\alpha_1 e_\lambda^{S_n} = e_\lambda^{S_n} e_1 = a_1e_{T_1}$. Therefore $e_{T_1}=a_1e_{T_1}$. With similar approach, inductively it can be proved that, for any $k=2,\ldots,d_\lambda$,
$$a_ke_{T_k}= (1-e_{T_1})\ldots (1-e_{T_{k-1}})e_{T_k}.$$
So, recursively we obtain
\begin{equation*}
\begin{split}
e_\lambda^{S_n} = \sum_{i=1}^{d_\lambda}  a_ie_{T_i} =  1-(1-e_1)+\sum_{k=2}^{d_\lambda} (1-e_{T_1})\ldots (1-e_{T_{k-1}})e_{T_k}\\
= 1 - \prod_{i=1}^{d_\lambda} (1-e_{T_{i}})
\end{split}    
\end{equation*}
which proves the result.\end{proof}

\subsection{Decomposition and centrally primitive idempotents in $\mathbb{F}_qA_n$}
Let $\pi$ be a permutation in the alternating group $A_n\subseteq S_n$. Let $\lambda=(\lambda_1,\lambda_2,\ldots,\lambda_k)$ be the partition of $n$ associated to the conjugation class of $\pi$ in $S_n$ and $\lambda^\prime$ be its conjugate partition. Denote by $\mathcal{C}_\pi$ the conjugation class of $\pi$ in $A_n$. There are only two possibilities for the conjugation class $\mathcal{C}_\lambda^{S_n}$ of $\pi$ in $S_n$: or it is equal to $\mathcal{C}_\pi$ (case 1) or it is equal to a disjoint union of two classes $\mathcal{C}_{\pi+}$ and $\mathcal{C}_{\pi-}$ in $A_n$ with equal cardinality (case 2). The case 2 occurs if and only if $\lambda_i$ is an odd number for all $i\in\{1,2,\ldots,k\}$ and $\lambda_1>\lambda_2>\ldots>\lambda_k$. In this case, denote by $\mathcal{C}_{\pi+}$ the conjugation class of $A_n$ containing the permutation $(1,\ldots,\lambda_1)(\lambda_{1}+1,\ldots,\lambda_1+\lambda_2)\ldots(\sum_{i=1}^{k-1}\lambda_{i}+1,\ldots,n)$ and by $\mathcal{C}_{\pi-}=\mathcal{C}_\lambda^{S_n}\backslash \mathcal{C}_{\pi+}$. Since each partition $\lambda$ of $n$ is biunivocally associated to an irreducible character $\chi_\lambda^{S_n}$ in $S_n$, the restriction of this character to $A_n$ occurs is two cases too. In the first case, if $\lambda\neq \lambda^\prime$, then there is a unique irreducible $A_n$-character $\chi_\lambda$ associated to $\lambda$ and $\chi_\lambda=\chi_\lambda^{S_n}|_{A_n}$, which is irreducible. In the second case, if $\lambda=\lambda^\prime$, there are two irreducible $A_n$-characters $\chi_{\lambda+}$ and $\chi_{\lambda-}$ associated to $\lambda$ and the characters is given by the following:

\begin{teorema}[\cite{james}, Theorem 2.5.7]\label{theoremAn} Consider $\lambda=(\lambda_1,\lambda_2,\ldots,\lambda_k)=\lambda^\prime \in P(n)$ and $h(\lambda):=(h_{11}(\lambda),\ldots,h_{rr}(\lambda))$, where $h_{ii}(\lambda)$ is the hook number associated to the $i$-th diagonal of the symmetric Young tableau $D_\lambda$. Consider $\alpha \in A_n$ and $\mu\in P(n)$ the partition associated to $\alpha$. So:
	\begin{enumerate}
		\item If $\mu\neq h(\lambda)$, then $\chi_{\lambda\pm}(\alpha) = \chi_\lambda^{S_n}(\alpha)/2.$
		\item If $\mu= h(\lambda)$ and $\alpha \in \mathcal{C}_{\alpha+}$, $$\chi_{\lambda\pm}(\alpha)=\frac{1}{2}\left((-1)^{(n-r)/2}\pm \sqrt{(-1)^{(n-r)/2}\prod_{i=1}^r h_{ii}(\lambda)}\right).$$
		\item If $\mu= h(\lambda)$ and $\alpha \in \mathcal{C}_{\alpha-}$, $$\chi_{\lambda\pm}(\alpha)=\frac{1}{2}\left((-1)^{(n-r)/2}\mp \sqrt{(-1)^{(n-r)/2}\prod_{i=1}^r h_{ii}(\lambda)}\right).$$
	\end{enumerate}
\end{teorema}

\begin{obs} It follows from of \cite[Theorems 1.2.17 and 2.4.3]{james} that, for $\lambda\in P(n)$, $\chi_{\lambda}^{S_n}(\alpha)\in\mathbb{Z}$ for every $\alpha\in A_n$ and $\chi_\lambda^{S_n}(1)=d_\lambda$.
\end{obs}

The last theorem allows us to obtain the Wedderburn-Artin decomposition of $\mathbb{F}_qA_n$. For this, consider
$$\Gamma = \{\{\lambda,\lambda^\prime\}~:~\lambda\in P(n), \lambda\neq \lambda^\prime \}$$
and 
$$\Delta = \left\{\lambda=\lambda^\prime \in P(n)~:~ \sqrt{p_\lambda} \in\mathbb{F}_q\right\}$$
where $h(\lambda)=(h_{11}(\lambda),\ldots,h_{rr}(\lambda))$ when $\lambda=\lambda^\prime$ and \begin{equation*}p_\lambda=(-1)^{(n-r)/2}\prod_{i=1}^r h_{ii}(\lambda).\end{equation*}
So, the Wedderburn-Artin decomposition of $\mathbb{F}_qA_n$ is given by
\begin{equation}\label{decAn}
\mathbb{F}_qA_n = \left(\bigoplus_{\gamma=\{\lambda,\lambda^\prime\} \in \Gamma} I_\gamma \right)\oplus \left(\bigoplus_{\lambda=\lambda^\prime \not\in \Delta} I_\lambda \right)\oplus \left(\bigoplus_{\lambda\in \Delta} (I_{\lambda+}\oplus I_{\lambda-}) \right),
\end{equation}
where
\begin{itemize}
	\item $I_\gamma \simeq M_{d_\lambda}(\mathbb{F}_q)$, if $\gamma=\{\lambda,\lambda^\prime\} \in \Gamma$;
	\item $I_\lambda \simeq M_{d_\lambda/2}(\mathbb{F}_{q^2})$, if $\lambda=\lambda^\prime\not\in \Delta$;
	\item $I_{\lambda\pm} \simeq M_{d_\lambda/2}(\mathbb{F}_q)$, if $\lambda=\lambda^\prime \in \Delta$.
\end{itemize}

\begin{obs} The Wedderburn-Artin decomposition of $\mathbb{Q}A_n$ was stated in \cite[Theorem 2]{giambruno} as a consequence of Theorem \ref{theoremAn}. So, the decomposition of $\mathbb{F}_qA_n$ given in (\ref{decAn}) also can be obtained as consequence of results presented in \cite[Section 3]{weintraub}.
\end{obs}

Usually, a complete set $\{e_i\}$ of centrally primitive idempotents of $\mathbb{F}_qA_n$ can be calculated using the character table of $A_n$ by the formula
$$e_{i} = \frac{\chi_i(1)}{n!/2}\sum_{g\in A_n} \chi_i(g^{-1})g,$$
for each character $\chi_i$ of $A_n$. In addition, computational tools can also help with this task. Expanding the possibilities, in the following theorem we present the centrally primitive idempotents associated to the components $I_\gamma$ of $\mathbb{F}_qA_n$ in the decomposition (\ref{decAn}), where $\gamma=\{\lambda,\lambda^\prime\}$ if $\lambda\neq \lambda^\prime$. This result and its proof are adapted from \cite[Theorem 4]{giambruno}, where the authors calculate the centrally primitive idempotents of $\mathbb{Q}A_n$ for the case $\lambda\neq \lambda^\prime$. In \cite[Theorem 5]{giambruno} it is given the centrally primitive idempotents associated to the other minimal ideals in the decomposition of $\mathbb{Q}A_n$, but the techniques used there can not be adapted to our case.

\begin{teorema}For each $\gamma=\{\lambda,\lambda^\prime \}\in \Gamma$, a centrally primitive idempotent that generates $I_\gamma$ is given by
\begin{equation}
e_\lambda := e_\lambda^{S_n}+e_{\lambda^\prime}^{S_n}.
\end{equation}
\end{teorema}
\begin{proof} Because of the proof of Theorem \ref{lemaorto} we know that
$$e_\lambda^{S_n} = \sum_{i=1}^{d_\lambda}  a_ie_{T_i}$$
where $a_1e_{T_1}=e_{T_1}$ and $a_ke_{T_k} = (1-e_{T_1})\ldots (1-e_{T_{k-1}})e_{T_k} \in \mathbb{F}_qS_n$, with $k=1,\ldots,d_\lambda$. Analogously, 
$$e_{\lambda^\prime}^{S_n} = \sum_{i=1}^{d_\lambda} \overline{a}_i\overline{e}_{T_i^\prime}$$
where $\overline{a}_1\overline{e}_{T_1^\prime}=\overline{e}_{T_1^\prime}$ and $\overline{a}_k\overline{e}_{T_k^\prime} = (1-\overline{e}_{T_1^\prime})\ldots (1-\overline{e}_{T_{k-1}^\prime})\overline{e}_{T_k^\prime} \in \mathbb{F}_qS_n$, with $k=1,\ldots,d_\lambda$. Since $T_k$ and $T_k^\prime$ are conjugated tableaux, for all $k$, then $R_{T_k}=C_{T_k^\prime}$ and $R_{T_k^\prime}=C_{T_k}$ and so $a_ke_{T_k}+\overline{a}_k\overline{e}_{T_k^\prime}$ has zero coefficient multiplying every odd permutation. This proves that $a_ke_{T_k}+\overline{a}_k\overline{e}_{T_k^\prime}\in \mathbb{F}_qA_n$ for $k=1,\ldots,d_\lambda$ and so $e_\lambda^{S_n}+e_{\lambda^\prime}^{S_n}\in \mathbb{F}_qA_n$.

Now we claim that $\Omega = \{a_1e_{T_1}+\overline{a}_1\overline{e}_{T_1^\prime},\ldots, a_{d_\lambda}e_{T_{d_\lambda}}+\overline{a}_{d_\lambda}\overline{e}_{T_{d_\lambda}^\prime}\}$ is a set of pairwise orthogonal primitive idempotents. Indeed, for each $k=1,\ldots,d_\lambda$, the projection $f:\mathbb{F}_qA_n(a_ke_{T_k}+\overline{a}_k\overline{e}_{T_k^\prime})\longrightarrow \mathbb{F}_qS_n e_{T_k}$ is an epimorphism of $\mathbb{F}_qA_n$-modules once $\chi_\lambda^{S_n}|_{A_n}$ is irreducible (and, so, $\mathbb{F}_qS_n e_{T_k}$ is irreducible over $\mathbb{F}_qA_n$). Since $\mathbb{F}_qS_n\overline{e}_{T_k^\prime}$ is also an irreducible $\mathbb{F}_qA_n$-module, if we suppose that $f$ is not an injection, then $\mathbb{F}_qS_n \overline{e}_{T_k^\prime} =\mathbb{F}_qA_n \overline{a}_k\overline{e}_{T_k^\prime} \subseteq \mathbb{F}_qA_n (a_ke_{T_k}+\overline{a}_k\overline{e}_{T_k^\prime})\subseteq \mathbb{F}_qA_n$. This implies that $\overline{e}_{T_k^\prime} \in \mathbb{F}_qA_n$. If this occurs, for every $p\in R_{T_k}$ and $q\in C_{T_k}$ such that $qp\in S_n\backslash A_n$, there exist $p^\prime\in R_{T_k}$ and $q^\prime \in C_{T_k}$ satisfying $qp=q^\prime p^\prime$ and $sgn(p)=-sgn(p^\prime)$. So $(q^\prime)^{-1}q=p^\prime p^{-1} \in R_{T_k}\cap C_{T_k}=\{1\}$. This implies $p^\prime = p$ and $q^\prime = q$, which is a contradiction. Therefore, $f$ is injective, that is, $f$ is an isomorphism of $\mathbb{F}_qA_n$-modules. This shows that $\mathbb{F}_qA_n (a_ke_{T_k}+\overline{a}_k\overline{e}_{T_k^\prime})$ is $\mathbb{F}_qA_n$-irreducible. Then the elements of $\Omega$ are primitive idempotents. All that remains is to show that the elements of $\Omega$ are pairwise orthogonal. In fact, since $e_{T_i}e_{T_j}=0$ if $i>j$ and each $a_i e_{T_i}$ is primitive in $\mathbb{F}_qS_n$ for all $i$, then $e_{T_i} a_{i} e_{T_i}=e_{T_i}$ and $a_ke_{T_k}=a_k e_{T_k} e_\lambda^{S_n} = \sum_{k=1}^{d_\lambda} (a_ie_i)(a_ke_k)$. Hence $\sum_{k=1,k\neq i}^{d_\lambda} (a_i e_{T_i})(a_k e_{T_k}) = 0$ implies $(a_ie_{T_i})(a_j e_{T_j}) = 0$ for $i\neq j$. Analogously $(\overline{a}_i\overline{e}_{T_i^\prime})(\overline{a}_j\overline{e}_{T_j^\prime})=0$ if $i\neq j$. Consequently $(a_ie_{T_i}+\overline{a}_i\overline{e}_{T_i^\prime})(a_je_{T_j}+\overline{a}_j\overline{e}_{T_j^\prime}) = 0$ for all $i\neq j$. This confirms that $\Omega$ is a set of pairwise orthogonal primitive idempotents.

Finally, since $\lambda\neq \lambda^\prime$, then $e_\lambda=e_\lambda^{S_n}+e_{\lambda^\prime}^{S_n}$ is a central idempotent of $\mathbb{F}_qA_n$ which is a sum of $\chi_\lambda^{A_n}(1)=\chi_\lambda^{S_n}(1)=d_\lambda$ orthogonal primitive idempotents (the elements of $\Omega$). This shows that $e_\lambda$ is a centrally primitive idempotent associated to $I_\gamma$, where $\gamma=\{\lambda,\lambda^\prime\}$.
\end{proof}

\begin{obs} In the case where $\lambda=\lambda^\prime\not\in \Delta$ the central idempotent $e_{\lambda}^{S_n}$ belongs to $\mathbb{F}_qA_n$ once $e_\lambda^{S_n}=e_{\lambda^\prime}^{S_n}$. So we can check whether this idempotent is primitive or not calculating the dimension of the ideal $\mathbb{F}_qA_n e_\lambda^{S_n}$ over $\mathbb{F}_q$. If this dimension is equal to $d_\lambda$, then $e_\lambda^{S_n}$ is the centrally primitive idempotent associated to $I_\lambda$, that is, the generator of $I_\lambda$.
\end{obs}

\section{Minimal group codes in $\mathbb{F}_qA_n$}\label{sec3}

Let $\mathbb{F}_q$ be the finite field of cardinality $q=p^f$, where $p$ is a prime number and $f\geq 1$ is a natural number. Let $G$ be a finite group such that $p$ does not divides $|G|$, which means that $\mathbb{F}_qG$ is semisimple. Consider $H$ a subgroup of $G$. If $|H|$ is invertible in $\mathbb{F}_q$, then
$\hat{H}:=|H|^{-1}\sum_{h\in H} h$ is an idempotent in $\mathbb{F}_qG$. Additionally, if $H\neq \{1\}$ is a normal subgroup of $G$, then $\hat{H}$ is central \cite[Lemma 3.6.6]{milies2002}. In the latter case, $\hat{H}$ is a sum of centrally primitive idempotents, which are called its \textit{constituents}. If $e$ is a constituent of $H$, then $e\hat{H}=e$ and the code $\mathbb{F}_qGe$ is a repetition code (see \cite{idem17}). In turn, if $e$ is a not a constituent of $H$, then $e\hat{H}=0$, which can produce more interesting codes like simplex binary codes of dimension $k$ and length $n=2^k-1$ (see \cite[Theorem 4.4]{idem17}). This motivates the following definition:

\begin{defin}[\cite{idem17}, Definition 2.2] A centrally primitive idempotent $e\in\mathbb{F}_qG$ is an \textbf{essential idempotent} if $e\hat{H}=0$, for every subgroup $H\neq \{1\}$ of $G$. A minimal ideal of $\mathbb{F}_qG$ is an \textbf{essential ideal} if it is generated by an essential idempotent.
\end{defin}

In the context of coding theory, we can say that a minimal code $C$ in $\mathbb{F}_qG$ is an \textbf{essential code} if $C$ is an essential ideal of $\mathbb{F}_qG$. In \cite[Proposition 2.3]{idem17} it is shown that a centrally primitive idempotent $e$ of $\mathbb{F}_qG$ is essential if and only if the map $\phi: G \longrightarrow Ge$ given by $\phi(g)=ge$, for all $g\in G$, is an isomorphism. This implies the following theorem:

\begin{teorema}\label{teoremaessen} If $G$ is a simple group and the characteristic of $\mathbb{F}_q$ does not divide the order of $G$, then every minimal code of $\mathbb{F}_qG$ is essential.
\end{teorema}
\begin{proof} Let $I\neq\{0\}$ be a minimal ideal of $\mathbb{F}_qG$. Since $\mathbb{F}_qG$ is semisimple, there exists a centrally primitive idempotent $e$ such that $I=\mathbb{F}_qGe$. Obviously the map $\phi:G\longrightarrow Ge$ given by $\phi(g)=ge$, for all $g\in G$, is an epimorphism. Since $G$ is simple, then $\phi$ is an isomorphism. Then, by the fact mentioned in latter paragraph, $I$ is an essential ideal.
\end{proof}

From now on consider $q=p^f$, where $p$ is a prime number and $f\geq 1$ is a positive integer number, and $n<p$ a positive integer number. In this case, $\mathbb{F}_qS_n$ and $\mathbb{F}_qA_n$ are semisimple group algebras. Theorem \ref{teoremaessen} implies the following evident consequence:

\begin{corol}\label{corolessen} If $n=3$ or $n\geq 5$, then every minimal code of the semisimple group algebra $\mathbb{F}_qA_n$ is essential.
\end{corol}
\begin{proof} It follows directly from Theorem \ref{teoremaessen}, once $A_n$ is simple for $n=3$ or $n\geq 5$.
\end{proof}

Note that Corollary \ref{corolessen} guarantees that all ideals $I_\gamma$, $I_\lambda$ and $I_{\lambda\pm}$ in the decomposition (\ref{decAn}) are essential codes of $\mathbb{F}_qA_n$ for $n=3$ or $n\geq 5$.

In the following we present some examples of codes in $\mathbb{F}_qA_n$ with hte calculations made in software \textit{SageMath} \cite{sage}.

\begin{ex} Consider $q=5$ and $n=3$. A complete set of centrally primitive idempotents of $\mathbb{F}_5S_3$ is given by $\{e_1,e_2,e_3\}$, where 
\begin{equation*}
\begin{split}
e_1^{S_3}=()+(2,3)+(1,2)+(1,2,3)+(1,3,2)+(1,3),    \\
e_2^{S_3}=4() + 3(1,2,3) + 3(1,3,2),\\
e_3^{S_3}=() + 4(2,3) + 4(1,2) + (1,2,3) + (1,3,2) + 4(1,3).
\end{split}
\end{equation*}
The idempotents $e_1$, $e_2$ and $e_3$ are associated, respectively, to the partitions $\lambda_1=(3)$, $\lambda_2=(2,1)$ and $\lambda_1^\prime = (1,1,1)$. Then
$$\mathbb{F}_5S_3 = \mathbb{F}_5 S_3 e_1^{S_3}\oplus \mathbb{F}_5 S_3 e_2^{S_3} \oplus \mathbb{F}_5 S_3 e_3^{S_3} \cong \mathbb{F}_5\oplus M_2(\mathbb{F}_5)\oplus \mathbb{F}_5.$$
The codes $\mathbb{F}_5S_3 e_1^{S_3}$ and $\mathbb{F}_5 S_3 e_3^{S_3}$ are repetition codes (they are $5$-ary $[6,1,6]$-codes) and the minimal code $\mathbb{F}_5S_3 e_2^{S_3}$ is a $5$-ary $[6,4,2]$-code. In turn, the Wedderburn-Artin decomposition of $\mathbb{F}_5A_3$ is given by
$$\mathbb{F}_5A_3 = \mathbb{F}_5A_3 e_1 \oplus \mathbb{F}_5A_3 e_2 \cong \mathbb{F}_5 \oplus \mathbb{F}_{5^2}.$$
A set of centrally primitive idempotents of this algebra is given by $\{e_1,e_2\}$, where $e_1=e_1^{S_3}+e_3^{S_3}$ and $e_2 = e_2^{S_3}$, once $\mathbb{F}_5A_3e_2$ has dimension $2=d_{\lambda_2}$. The code $\mathbb{F}_5A_3 e_1$ is a repetition code (it is a $5$-ary $[3,1,3]$-code) and the code $\mathbb{F}_5A_3e_2$ is a $[3,2,2]$, which is the best linear code with length $n=3$ and dimension $k=2$ over $\mathbb{F}_5$ \cite{codetable}. The generator matrix of the last code is given by
$$\left(\begin{array}{rrr}
1 & 0 & 4 \\
0 & 1 & 4
\end{array}\right).$$
\end{ex}

\begin{ex} For $q=7$ and $n=3$, the Wedderburn-Artin decomposition of $\mathbb{F}_7S_3$ and $\mathbb{F}_7A_3$ are given, respectively, by
$$\mathbb{F}_7 S_3 = \mathbb{F}_7 S_3 e_1^{S_3} \oplus \mathbb{F}_7 S_3 e_2^{S_3}\oplus \mathbb{F}_7 S_3 e_3^{S_3} \cong \mathbb{F}_7\oplus M_2(\mathbb{F}_7)\oplus \mathbb{F}_7$$
and
$$\mathbb{F}_7 A_3 = \mathbb{F}_7 e_1\oplus \mathbb{F}_7 e_2 \oplus \mathbb{F}_7 e_3 \cong \mathbb{F}_7\oplus \mathbb{F}_7\oplus \mathbb{F}_7,$$
where
\begin{equation*}
\begin{split}
e_1^{S_3} = 6() + 6(2,3) + 6(1,2) + 6(1,2,3) + 6(1,3,2) + 6(1,3),    \\
e_2^{S_3}=3() + 2(1,2,3) + 2(1,3,2),\\
e_3^{S_3}=6() + (2,3) + (1,2) + 6(1,2,3) + 6(1,3,2) + (1,3),\\
e_1 = e_1^{S_3}+e_3^{S_3},\\
e_2 = 5() + 6(1,2,3) + 3(1,3,2),\\
e_3 = 5() + 3(1,2,3) + 6(1,3,2).
\end{split}
\end{equation*}
The codes $\mathbb{F}_7 S_3 e_1^{S_3}$ and $\mathbb{F}_7 S_3 e_1^{S_3}$ are trivial $7$-ary $[6,1,6]$-codes and all the minimal codes of $\mathbb{F}_7A_3$ are trivial $7$-ary $[3,1,3]$-codes. The code $\mathbb{F}_7S_3 e_2^{S_3}$ is a $7$-ary $[6,4,2]$-code.
\end{ex}

\begin{ex} Consider $n=4$ and $q=5$. In this case, the Wedderburn-Artin decomposition of $\mathbb{F}_5A_4$ is given by
$$\mathbb{F}_5 A_4 = \mathbb{F}_5 A_4 e_1 \oplus \mathbb{F}_5A_4 e_2 \oplus \mathbb{F}_5 A_4e_3 \cong \mathbb{F}_5\oplus M_3(\mathbb{F}_5)\oplus \mathbb{F}_{25}$$
where $e_1=3 + 3(2,3,4) + 3(2,4,3) + 3(1,2)(3,4) + 3(1,2,3) + 3(1,2,4) + 3(1,3,2) + 3(1,3,4) + 3(1,3)(2,4) + 3(1,4,2) + 3(1,4,3) + 3(1,4)(2,3)$ is the idempotent associated to the partition $\lambda_1=(4)$, the idempotent $e_2 = 2 + (1,2)(3,4) + (1,3)(2,4) + (1,4)(2,3)$ is associated to the partition $\lambda_2=(3,1)$ and $e_3 = () + 2(2,3,4) + 2(2,4,3) + (1,2)(3,4) + 2(1,2,3) + 2(1,2,4) + 2(1,3,2) + 2(1,3,4) + (1,3)(2,4) + 2(1,4,2) + 2(1,4,3) + (1,4)(2,3)$ is associated to the partition $\lambda_3 = (2,2)$. The minimal code $\mathbb{F}_5 A_4 e_1$ is a repetition $5$-ary code $[3,1,3]$, the minimal code $\mathbb{F}_5 A_4 e_2$ is a $5$-ary $[12,9,2]$-code and the minimal code $\mathbb{F}_5 A_4 e_3$ is a $5$-ary $[12,2,8]$-code. The generator matrices of the last two codes are given, respectively, by
$$M_2=\left(\begin{array}{rrrrrrrrrrrr}
0 & 0 & 0 & 0 & 0 & 0 & 0 & 0 & 0 & 0 & 1 & 4 \\
0 & 0 & 0 & 0 & 0 & 0 & 0 & 0 & 0 & 1 & 0 & 4 \\
0 & 0 & 0 & 0 & 0 & 0 & 0 & 0 & 1 & 0 & 0 & 4 \\
0 & 0 & 0 & 0 & 0 & 0 & 1 & 4 & 0 & 0 & 0 & 0 \\
0 & 0 & 0 & 0 & 0 & 1 & 0 & 4 & 0 & 0 & 0 & 0 \\
0 & 0 & 0 & 0 & 1 & 0 & 0 & 4 & 0 & 0 & 0 & 0 \\
0 & 0 & 1 & 4 & 0 & 0 & 0 & 0 & 0 & 0 & 0 & 0 \\
0 & 1 & 0 & 4 & 0 & 0 & 0 & 0 & 0 & 0 & 0 & 0 \\
1 & 0 & 0 & 4 & 0 & 0 & 0 & 0 & 0 & 0 & 0 & 0
\end{array}\right)$$
and
$$M_3 = \left(\begin{array}{rrrrrrrrrrrr}
1 & 1 & 1 & 1 & 0 & 0 & 0 & 0 & 4 & 4 & 4 & 4 \\
0 & 0 & 0 & 0 & 1 & 1 & 1 & 1 & 4 & 4 & 4 & 4
\end{array}\right).$$
We remark that the best minimum distance for $5$-ary codes with length $n=12$ and dimension $k=9$ is $d=3$ and that the best minimum distance for $5$-ary codes with $n=12$ and $k=2$ is $d=10$ \cite{codetable}.
\end{ex}

\begin{ex} Consider $n=4$ and $q=7$. The Wedderburn-Artin decomposition in this case is given by
$$\mathbb{F}_7 A_4 = \mathbb{F}_7 A_4 e_1\oplus \mathbb{F}_7 A_4 e_2\oplus \mathbb{F}_7 A_4 e_3 \oplus \mathbb{F}_7 A_4 e_4 \cong \mathbb{F}_7\oplus M_3(\mathbb{F}_7)\oplus \mathbb{F}_7\oplus \mathbb{F}_7$$
where $e_1=3 + 3(2,3,4) + 3(2,4,3) + 3(1,2)(3,4) + 3(1,2,3) + 3(1,2,4) + 3(1,3,2) + 3(1,3,4) + 3(1,3)(2,4) + 3(1,4,2) + 3(1,4,3) + 3(1,4)(2,3)$ is the idempotent associated to $\lambda_1=(4)$, the idempotent $e_2=6 + 5(1,2)(3,4) + 5(1,3)(2,4) + 5(1,4)(2,3)$ is associated to $\lambda_2 = (3,1)$ and the idempotents $e_3=3 + 5(2,3,4) + 6(2,4,3) + 3(1,2)(3,4) + 6(1,2,3) + 5(1,2,4) + 5(1,3,2) + 6(1,3,4) + 3(1,3)(2,4) + 6(1,4,2) + 5(1,4,3) + 3(1,4)(2,3)$ and $e_4=3 + 6(2,3,4) + 5(2,4,3) + 3(1,2)(3,4) + 5(1,2,3) + 6(1,2,4) + 6(1,3,2) + 5(1,3,4) + 3(1,3)(2,4) + 5(1,4,2) + 6(1,4,3) + 3(1,4)(2,3)$ are the idempotents associated to the partition $\lambda_3=(2,2)$ (observe that $\lambda_3\in \Delta$). The minimal codes $\mathbb{F}_7 A_4 e_i$, for $i=1,3,4$, are $7$-ary trivial codes $[4,1,4]$. In turn, the minimal code $\mathbb{F}_7 A_4 e_2$ has generator matrix given by
$$
\left(\begin{array}{rrrrrrrrrrrr}
0 & 0 & 0 & 0 & 0 & 0 & 0 & 0 & 0 & 0 & 1 & 6 \\
0 & 0 & 0 & 0 & 0 & 0 & 0 & 0 & 0 & 1 & 0 & 6 \\
0 & 0 & 0 & 0 & 0 & 0 & 0 & 0 & 1 & 0 & 0 & 6 \\
0 & 0 & 0 & 0 & 0 & 0 & 1 & 6 & 0 & 0 & 0 & 0 \\
0 & 0 & 0 & 0 & 0 & 1 & 0 & 6 & 0 & 0 & 0 & 0 \\
0 & 0 & 0 & 0 & 1 & 0 & 0 & 6 & 0 & 0 & 0 & 0 \\
0 & 0 & 1 & 6 & 0 & 0 & 0 & 0 & 0 & 0 & 0 & 0 \\
0 & 1 & 0 & 6 & 0 & 0 & 0 & 0 & 0 & 0 & 0 & 0 \\
1 & 0 & 0 & 6 & 0 & 0 & 0 & 0 & 0 & 0 & 0 & 0
\end{array}\right)
$$
and corresponds to a $7$-ary $[12,9,2]$-code. We remark that the best minimum distance for $7$-ary codes with length $n=12$ and dimension $k=9$ is $d=3$ \cite{codetable}.
\end{ex}

\end{document}